\documentclass[12pt]{article}

\usepackage{pifont}
\usepackage{amssymb}
\usepackage{amsmath}
\usepackage{arydshln}
\usepackage{fancyhdr}
\usepackage{setspace}
\usepackage{enumitem}
\setlist{nolistsep}
\usepackage{graphicx}
\usepackage[usenames,dvipsnames]{color}
\usepackage{epsfig}
\usepackage[mathscr]{eucal}
\newtheorem{theorem}{Theorem}[section]
\newtheorem{lemma}[theorem]{Lemma}
\newtheorem{corollary}[theorem]{Corollary}

\newtheorem{remark}[theorem]{Remark}

\newtheorem{result}[theorem]{Result}
\newenvironment{proof}{\noindent{\bf Proof}\hspace{0.5em}}
    { \null  \hfill $\square$ \par}

\newcommand\B{{\mathscr B}}

\newcommand\C{{\cal C}}

\newcommand\V{{\cal V}}
\newcommand\var{{{\mathcal V}^5_2}}

\newcommand\N{{\cal N}}

\renewcommand{\S}{\mathcal S}

\newcommand{\K}{\mathcal K}

\newcommand{\si}{\Sigma_\infty}

\newcommand{\li}{\ell_\infty}

\renewcommand\setminus{\backslash}
\newcommand{\st}{\,|\,}

\newcommand\PGL{{\rm PGL}}

\newcommand\PG{{\rm PG}}
\newcommand\AG{{\rm AG}}

\newcommand{\Label}{\label}

    \newcommand{\Fq}{\mathbb F_{q}}

\renewcommand{\lambda}{t}

\renewcommand{\star}{{^{\mbox{\tiny\ding{73}}}}}

\begin{document}
%
%

\title{A ruled quintic surface in $\PG(6,q)$}
\author{S.G. Barwick}

\maketitle

AMS code: 51E20

keywords: projective space, varieties, scroll, Bruck-Bose representation

\abstract{In this article we look at a scroll of $\PG(6,q)$ that uses a projectivity to rule a conic and a twisted cubic. We show this scroll is a ruled quintic surface $\var$, and study its geometric properties. The motivation in studying this scroll lies in its relationship with an $\Fq$-subplane of $\PG(2,q^3)$ via the Bruck-Bose representation.}

\section{Introduction}

In this article we consider a scroll of $\PG(6,q)$ that rules a conic and a twisted cubic according to a projectivity. The motivation in studying this scroll lies in its relationship with an $\Fq$-subplane of $\PG(2,q^3)$ via the Bruck-Bose representation as described in Section~\ref{sec-BB}.
In $\PG(6,q)$, let $\C$ be a non-degenerate conic in a plane $\alpha$, $\C$ is called the {\em conic directrix}. Let $\N_3$ be a twisted cubic in a 3-space $\Pi_3$ with $\alpha\cap\Pi_3=\emptyset$, $\N_3$ is called the {\em twisted cubic directrix}.
 Let $\phi$ be a projectivity from the points of $\C$ to the points of $\N_3$. 
By this we mean that if we write the points of $\C$ and $\N_3$ using  a non-homogeneous parameter, so $\C=\{C_\theta=(1,\theta,\theta^2)\st\theta\in\Fq\cup\{\infty\}\}$, and $\N_3=\{N_\epsilon=(1,\epsilon,\epsilon^2,
\epsilon^3)\st\epsilon\in\Fq\cup\{\infty\}\}$, then $\phi\in\PGL(2,q)$ is a projectivity mapping $(1,\theta)$ to $(1,\epsilon)$. 
Let $\V$ be the set of points of $\PG(6,q)$ lying on the $q+1$ lines  joining each point of $\C$ to the corresponding point (under $\phi$) of $\N_3$.  These $q+1$ lines are called the {\em generators} of $\V$. As the two subspaces $\alpha$ and $\Pi_3$ are disjoint,  $\V$ is not contained in a 5-space.
  We note that this construction generalises the ruled cubic surface $\V^3_2$ in $\PG(4,q)$, a variety that has been well studied, see \cite{vincenti}. 
    
We will be working with normal rational curves in $\PG(6,q)$, if $\N$ is a normal rational curve that generates an $i$-dimensional space, then we call $\N$ an  {\em $i$-dim nrc}, and often use the notation $\N_i$. 
See \cite{HT} for details on normal rational curves. 
As  we will be looking at   5-dim nrcs contained in $\V$, 
we will assume $q\geq 6$ throughout. 

This article studies the geometric structure of $\V$. 
In Section~\ref{sec-simple}, we show that $\V$ is a variety $\var$ of order 5 and dimension 2, and that all such scrolls are projectively equivalent. Further, we show that $\V$ contains exactly $q+1$ lines and one non-degenerate conic. 
In Section~\ref{sec-BB},  we describe the Bruck-Bose representation of $\PG(2,q^3)$ in $\PG(6,q)$, and discuss how $\V$ corresponds to an $\Fq$-subplane of $\PG(2,q^3)$. We use the Bruck-Bose setting to 
 show that $\V$ contains exactly $q^2$ twisted cubics, and that each can act as a directrix of $\V$. In Section~\ref{sec-5}, we count the number of 4 and 5-dim nrcs contained in $\V$. Further, 
we determine how 5-spaces meet $\V$, and count the number of 5-spaces of each intersection type. The main result is Theorem~\ref{count-nrc}.
In Section~\ref{sec-BB-5}, we determine how 5-spaces meet $\V$ in relation to the regular 2-spread in the Bruck-Bose setting.

\section{Simple properties of $\V$}\Label{sec-simple}

\begin{theorem}\Label{the-nine-quad}
Let $\V$ be a scroll of $\PG(6,q)$ that rules a conic and a twisted cubic according to a projectivity. Then 
 $\V$ is a variety of dimension 2 and order 5, denoted $\var$ and called a ruled quintic surface. Further, any two ruled quintic surfaces are projectively equivalent. 
\end{theorem}

\begin{proof} Let $\V$ be a scroll of $\PG(6,q)$ with conic directrix $\C$ in a plane $\alpha$,  twisted cubic directrix $\N_3$ in a 3-space $\Pi_3$, and ruled by a projectivity as described in Section 1. 
The group of collineations of $\PG(6,q)$ is transitive on planes; and transitive on 3-spaces. Further, all non-degenerate conics in a projective plane are projectively equivalent; and all twisted cubics in a 3-space are projectively equivalent. Hence without loss of generality, we can coordinatise $\V$ as follows. Let $\alpha$ be the plane which is the intersection of the four hyperplanes
$ x_0=0,\ x_1=0,\ x_2=0,\ x_3=0.$
Let $\C$ be the non-degenerate conic in $\alpha$ with points $C_\theta=(0,0,0,0,1,\theta,\theta^2)$ for $\theta\in\mathbb F_q\cup\{\infty\}$. Note that the points of $\C$ are the exact intersection of $\alpha$ with the  quadric of equation $x_5^2=x_4x_6.$
Let $\Pi_3$ be the 3-space which is the intersection of the  three hyperplanes
$ x_4=0,\ x_5=0,\ x_6=0.$
Let $\N_3$ be the twisted cubic in $\Pi_3$ with points $N_\theta=(1,\theta,\theta^2,\theta^3,0,0,0)$ for $\theta\in\mathbb F_q\cup\{\infty\}$. Note that the points of $\N_3$ are the exact intersection of $\Pi_3$ with the three quadrics with equations $x_1^2=x_0x_2,\  x_2^2=x_1x_3, \  x_0x_3=x_1x_2.$
A projectivity in $\PGL(2,q)$ is uniquely determined by the image of three points, so 
without loss of generality, let $\V$ have generator lines  $\ell_\theta=\{V_{\theta,t}=N_\theta+ t C_\theta,\ t\in\mathbb F_q\cup\{\infty\}\}$ for $\theta\in\Fq\cup\{\infty\}$.
That is, $V_{\theta,t}=(1,\theta,\theta^2,\theta^3,t,t\theta,t\theta^2)$. Equivalently, $\V$ consists of the points $V_{x,y,z}=(x^3,\ x^2y,\ xy^2,\ y^3,\ zx^2,\ zxy,\ zy^2)$ for $ x,y\in\Fq$, not both 0, $z\in\Fq\cup\{\infty\}.$
It is straightforward to verify that the pointset of $\V$ is  the exact intersection of the following ten quadrics,
 $
     x_{0}x_{5} =x_{1}x_{4},\ 
    x_0x_6= x_1x_5=x_2x_4,\ 
    x_1x_6=x_2x_5=x_3x_4,\ 
    x_2x_6=x_3x_5,\ 
    x_1^2=x_0x_2,\ 
    x_2^2=x_1x_3,\ 
    x_5^2=x_4x_6,\ 
    x_0x_3=x_1x_2.
    $
Hence the points of $\V$ form a variety. 

We follow \cite{semple} to calculate the dimension and order of $\V$.
The following map  defines an algebraic one to one  correspondence between  the plane $\pi$ of $PG(3,q)$ with points $(x,y,z,0)$,  $x,y,z\in\Fq$ not all 0, and the points of $\V$. $$\begin{array}{rccl}
\sigma : &\pi&\longrightarrow&\V\\
                &           (x,y,z,0) & \longmapsto &  (x^3,x^2y,xy^2,y^3,x^2z,xyz,y^2z).
\end{array}$$
Thus $\V$ is an absolutely
irreducible variety of dimension two and so we are justified
in calling it a surface.
Now consider a generic 4-space of $\PG(6,q)$ with equation given by the two hyperplanes 
$\Sigma_1: a_0x_0+\cdots+a_6x_6=0$ and  
$\Sigma_2: b_0x_0+\cdots+b_6x_6=0$, $a_i,b_i\in\Fq$. 
The point $V_{x,y,z}=(x^3,x^2y,xy^2,y^3,x^2z,xyz,y^2z)$ lies on $\Sigma_1$ if $a_0x^3+a_1x^2y+ a_2xy^2+a_3y^3+a_4x^2z+a_5xyz+a_6y^2z=0$. This  corresponds to a cubic $\K$ in the plane  $\pi$, moreover $\K$   contains the point $P=(0,0,1,0)$, and $P$ is a double point of $\K$. 
 Similarly  $V_{x,y,z}\in\Sigma_2$ corresponds to a cubic in $\pi$ with a double point $(0,0,1,0)$. Two cubics in a plane meet generically in nine points. As $(0,0,1,0)$ lies in the kernel of $\sigma$,  in $\PG(6,q)$ the 4-space $\Sigma_1\cap\Sigma_2$ meets $\V$ in five points, and so $\V$ has order 5.
 \end{proof}

\begin{theorem}\Label{x-gen} Let $\var$ be a ruled quintic surface in $\PG(6,q)$. 
\begin{enumerate}
\item No two generators of $\var$ lie in a plane.
\item No three generators of $\var$ lie in a 4-space.
\item No four generators of $\var$  lie in a 5-space.
\end{enumerate}
\end{theorem}

\begin{proof} Let $\var$ be a ruled quintic surface of $\PG(6,q)$ with conic directrix $\C$ in a plane $\alpha$, and twisted cubic directrix  $\N_3$ lying in a 3-space $\Pi_3$. 
Suppose two generator lines $\ell_0,\ell_1$ of $\var$ lie in a plane.
Let $m$ be the line in $\alpha$ joining the distinct points $\ell_0\cap\alpha$, $\ell_1\cap\alpha$. Let  $m'$ be the line in $\Pi_3$ joining the distinct points $\ell_0\cap\Pi_3$, $\ell_1\cap\Pi_3$. The lines $m,m'$ lie in the plane $\langle\ell_0,\ell_1\rangle$ and so meet in a point, contradicting $\alpha$, $\Pi_3$ being disjoint. 
Hence the generator lines of $\var$ are pairwise skew.

For part 2, suppose a 4-space $\Pi_4$ contains three distinct generators of $\var$. As distinct generators meet $\C$ in distinct points, 
 $\Pi_4$ contains three distinct points of $\C$, and so contains the plane $\alpha$. Further, distinct generators meet $\N_3$ in distinct points, hence $\Pi_4$ contains three points of $\N_3$, and so $\Pi_4\cap\Pi_3$ has dimension at least two. Hence $\langle\Pi_4,\Pi_3\rangle$  has dimension at most $4+3-2=5$.
 However,  $\var\subseteq \langle\Pi_4,\Pi_3\rangle$, a contradiction as $\var$ is not contained in a 5-space. 

For part 3, suppose a 5-space $\Pi_5$ contains   four distinct generators of $\var$. Distinct generators meet $\Pi_3$ in distinct points of $\N_3$, so $\Pi_5$ contains four points of $\N_3$, which do not lie in a plane. Hence $\Pi_5$ contains $\Pi_3$. Similarly $\Pi_5$ contains $\alpha$, and so $\Pi_5$ contains $\var$,  a contradiction as $\var$ is not contained in a 5-space.
\end{proof}

\begin{corollary}\Label{2-gen-alpha}
 No two generators of $\var$ lie in a 3-space containing $\alpha$. 
\end{corollary}

\begin{proof} Suppose a 3-space $\Pi_3$ contained $\alpha$ and two generators of $\var$. Let $P$ be a point of $\var$ not in $\Pi_3$, and let $\ell$ be the generator of $\var$ through $P$.  Then $\Pi_4=\langle \Pi_3,P\rangle$ contains two distinct points of $\ell$, namely $P$ and $\ell\cap\C$, and so $\Pi_4$ contains $\ell$. That is, $\Pi_4$   is a 4-space containing three generators, contradicting Theorem~\ref{x-gen}.
\end{proof}

We now show that the only lines on $\var$ are the generators, and the only non-degenerate conic on $\var$ is the conic directrix. We show later in Theorem~\ref{tc-direct} that there are exactly $q^2$ twisted cubics on $\var$, and that each is a directrix.

\begin{theorem}\Label{line-rqs}
Let $\var$  be a ruled quintic surface in $\PG(6,q)$. A line of $\PG(6,q)$ meets $\var$ in $0$, $1$, $2$ or $q+1$ points. Further, $\var$ contains exactly $q+1$ lines, namely the generator lines.
\end{theorem}

\begin{proof}
  Let $\var$ be a ruled quintic surface of $\PG(6,q)$  with conic directrix $\C$ lying in a plane $\alpha$, and twisted cubic directrix $\N_3$ lying in the 3-space $\Pi_3$. Let $m$ be a line of $\PG(6,q)$ that is not a generator of $\var$, and 
suppose $m$ meets $\var$ in three points $P,Q,R$. As $m$ is not a generator of $\var$, the points  $P,Q,R$ lie on distinct generator lines denoted $\ell_P,\ell_Q,\ell_R$ respectively. 
As $\C$ is a non-degenerate conic, $m$ is not a line of $\alpha$ and so at most one of the points $P,Q,R$ lie in $\C$. 
Suppose firstly that $P,Q,R\notin\C$. Then $\langle\alpha,m\rangle$ is a 3- or 4-space that contains the three generators $\ell_P,\ell_Q,\ell_R$, contradicting Theorem~\ref{x-gen}. 
Now suppose $P\in\C$ and $Q,R\notin\C$. Then $\Sigma_3=\langle \alpha,m\rangle$ is a 3-space which contains the two generator lines $\ell_Q,\ell_R$. So $\Sigma_3\cap\Pi_3$ contains the distinct points $\ell_R\cap\N_3$, $\ell_Q\cap\N_3$, and so has dimension at least one. Hence $\langle\Sigma_3,\Pi_3\rangle$ has dimension at most $3+3-1=5$,   a contradiction as $\var\subset\langle\Sigma_3,\Pi_3\rangle$, but $\var$ is not contained in a 5-space. Hence a line of $\PG(6,q)$ is either a generator line of $\var$, or meets $\var$ in 0, 1 or 2 points. 
\end{proof}

\begin{theorem}\Label{conic-dir-unique}
The ruled quintic surface  $\var$ contains exactly one  non-degenerate conic. 
\end{theorem}

\begin{proof} Let $\var$ be a ruled quintic surface with conic directrix $\C$ in a plane $\alpha$. Suppose $\var$ contains another non-degenerate conic $\C'$ in a plane $\alpha'\neq\alpha$. If $\C'$ contains   two points on a generator $\ell$ of $\var$, then $\alpha'\cap\var$ contains $\C'$ and  $\ell$. 
 However, by the proof of Theorem~\ref{the-nine-quad}, $\var$ is the intersection of quadrics, and the configuration $\C'\cup\ell$ is not contained  in any planar quadric. Hence $\C'$ contains  exactly one point on each generator of $\var$. We consider the three cases where  $\alpha\cap\alpha'$ is either empty, a point or a line. 
 Suppose $\alpha\cap\alpha'=\emptyset$, then $\langle \alpha,\alpha'\rangle$ is a 5-space that contains $\C$ and $\C'$, and so contains two distinct points on each generator of $\var$. Hence  $\langle \alpha,\alpha'\rangle$ contains each generator of $\var$ and so contains $\var$, a contradiction as $\var$ is not contained in a 5-space. Suppose $\alpha\cap\alpha'$ is a point $P$, then  $\langle \alpha,\alpha'\rangle$ is a 4-space that contains at least $q$  generators of $\var$, contradicting Theorem~\ref{x-gen} as $q\geq 6$. Finally, suppose $\alpha\cap\alpha'$ is a line, then $\langle \alpha,\alpha'\rangle$ is a 3-space that contains at least $q-1$ generators, contradicting Theorem~\ref{x-gen}  as $q\geq 6$. So $\var$ contains exactly one non-degenerate conic. 
\end{proof}

We aim to classify how  5-spaces meet $\var$, so we begin with a simple description.

\begin{remark}\Label{remark-5} Let $\Pi_5$ be a 5-space, then $\Pi_5\cap\var$ contains a set of $q+1$ points, one on each generator. 
\end{remark}

\begin{lemma}\Label{5-space-quintic} 
A $5$-space meets $\var$ in either
(a) a 5-dim nrc,
(b)  a 4-dim nrc and 0 or 1 generators,
(c)  a 3-dim nrc and 0, 1 or 2  generators,
(d)   the conic directrix and 0, 1, 2 or 3 generators.
\end{lemma}
\begin{proof}
Using properties of varieties (see for example \cite{semple}) we have  $\var\cap\V^1_5=\V^5_1$, that is,
the variety $\var$ meets a 5-space $\V^1_5$ in a curve of degree five. Denote this curve of $\PG(6,q)$ by $\K$. The degree of $\K$ can be partitioned as $$5=4+1=3+2=3+1+1=2+2+1=2+1+1+1=1+1+1+1+1.$$ 
By Theorem~\ref{line-rqs}, the only lines on $\var$ are the generators. By Theorem~\ref{x-gen}, $\K$ does not contain more than 3 generators.  By Remark~\ref{remark-5}, $\K$ contains at least one point on each generator.  Hence $\K$ is not empty, and is not the union of 1, 2 or 3 generators, so the partition $1+1+1+1+1$ for the degree of $\K$ does not occur. 

Suppose that  the  degree of $\K$  is partitioned as either  (a) $2+2+1$ or (b) $2+1+1+1$. By  Remark~\ref{remark-5}, $\K$ contains a point on each generator, so $\K$ contains   an irreducible conic. By Theorem~\ref{conic-dir-unique}, this conic is the conic directrix $\C$ of $\var$, and case (a) does not occur. 
Hence $\K$ consists of $\C$ and 0, 1, 2 or 3 generators of $\var$.

Suppose that the  degree of $\K$ is partitioned as  $3+1+1$. So $\K$ consists of at most 2 generators, and an irreducible cubic $\K'$. 
By  Remark~\ref{remark-5}, $\K$ contains a point on each generator, so $\K'$ contains  a point on at least $q-1$ generators.
If $\K'$ generates a 3-space, then it is a 
 3-dim nrc of $\PG(6,q)$. If not, $\K'$ is an irreducible cubic  contained in a  plane $\Pi_2$.  By the proof of Theorem~\ref{the-nine-quad}, $\K'$ is contained in a quadric, so $\K'$ is not an irreducible planar cubic. Thus $\K'$ is a 3-dim nrc of $\PG(6,q)$. 
Hence $\K$  consists of  a 3-dim nrc  and  0, 1 or 2 generators of $\var$.

Suppose that  the  degree of $\K$  is partitioned as $2+3$.
By  Remark~\ref{remark-5}, $\K$ contains a point on each generator. As argued above, $\K$ does not contain an irreducible planar cubic. Suppose $\K$ contained both an irreducible conic $\C$ and a twisted cubic $\N_3$, then there is at least one generator $\ell$ that meet $\C$ and $\N_3$ in distinct points. In this case $\ell$ lies in the 5-space and so lies in $\K$, a contradiction. So $\K$ is not the union of an irreducible conic and a twisted cubic.

Suppose that the  degree of $\K$ is partitioned as  $4+1$. So $\K$ consists of at most 1 generator, and an irreducible quartic $\K'$.  By  Remark~\ref{remark-5}, $\K$ contains a point on each generator, so $\K'$ contains  a point on at least $q$ generators.
If $\K'$ generates a 4-space, then it is a 
 4-dim nrc of $\PG(6,q)$. If not, $\K'$ is an irreducible quartic contained in a  3-space $\Pi_3$.  Let $\ell,m$ be two  generators not in $\K$, then by Remark~\ref{remark-5} they meet $\K'$.  So $\langle\Pi_3,\ell,m\rangle$ has dimension at most 5, and meets $\var$ in a irreducible quartic and 2 lines, which is a curve of degree 6, a contradiction.   Thus $\K'$ is a 4-dim nrc of $\PG(6,q)$. 
That is, $\K$  consists of  a 4-dim nrc  and  0 or 1 generators of $\var$.

Suppose the curve $\K$ is irreducible.  By  Remark~\ref{remark-5}, $\K$ contains a point on each generator. So either $\K$ is a 5-dim nrc of $\PG(6,q)$, or $\K$ lies in a 4-space.  Suppose $\K$ lies in a 4-space $\Pi_4$, and let $\ell$ be a generator, then $\langle\Pi_4,\ell\rangle$ has dimension at most 5  and  meets $\var$ in a curve of degree 6, a contradiction. So $\K$ is a 5-dim nrc of $\PG(6,q)$. 
 \end{proof}

\begin{corollary}\Label{3-4-no-gen}
Let $\Pi_r$ be an $r$-space, $r=3,4,5$ that contains an $r$-dim nrc of $\var$. Then $\Pi_r$ contains 0 generators of $\var$.
\end{corollary}

\begin{proof}
First suppose $r=3$, by Lemma~\ref{5-space-quintic}, a 5-space containing a twisted cubic $\N_3$ of $\var$ contains at most two generators  of $\var$. Hence a 4-space containing $\N_3$ contains at most one generator of $\var$. Hence the 3-space $\Pi_3$ containing $\N_3$ contains no generator of $\var$. 
If $r=4$, by Lemma~\ref{5-space-quintic}, a 5-space containing a 4-dim nrc $\N_4$ of $\var$ contains at most one generator  of $\var$. Hence the 4-space $\Pi_4$ containing $\N_4$ contains no generators of $\var$. If $r=5$, then by Lemma~\ref{5-space-quintic}, $\Pi_5$ contains 0 generators of $\var$.
\end{proof}

\begin{theorem}\Label{one-pt-gen} Let $\N_r$ be an $r$-dim nrc lying on $\var$, $r=3,4,5$. Then 
$\N_r$ contains exactly  one point on each generator of $\var$.
\end{theorem}

\begin{proof} Let $\N_r$ be an $r$-dim nrc lying on $\var$, $r=3,4,5$, and denote the $r$-space containing $\N_r$ by $\Pi_r$. 
If $\Pi_r$ contained 2 points of a generator of $\var$, then it contains the whole generator, so by Corollary~\ref{3-4-no-gen}, the $q+1$ points of $\N_r$ are one on each generator of $\var$. 
 \end{proof}

\section{$\var$ and  $\Fq$-subplanes of $\PG(2,q^3)$}\Label{sec-BB}

To study $\var$ in more detail, we use the linear representation of $\PG(2,q^3)$ in $\PG(6,q)$ developed independently by Andr\'e and Bruck and Bose in \cite{BB0, BB1, BB2}. Let $\S$ be a regular 2-spread of $\PG(6,q)$ in a 5-space $\si$. Let $\mathscr I$ be the incidence structure with: {\sl points} the points of $\PG(6,q)\setminus\si$;  {\sl lines} the 3-spaces of $\PG(6,q)$ that contain a plane of $\S$ and  are not in $\si$; and {\sl incidence} is inclusion. Then  $\mathscr I$   is isomorphic to $\AG(2,q^3)$. We can uniquely complete  $\mathscr I$  to $\PG(2,q^3)$, the points on $\li$ correspond to the planes of $\S$. We call this the {\em Bruck-Bose representation} of $\PG(2,q^3)$ in $\PG(6,q)$, see \cite{FFA} for a detailed discussion on this representation. 
Of particular interest is the relationship between the ruled quintic surface of $\PG(6,q)$ and the $\Fq$-subplanes of $\PG(2,q^3)$. 

To describe this relationship, we 
 need to use the cubic extension of $\PG(6,q)$ to $\PG(6,q^3)$. The regular 2-spread $\S$ has a unique set of three conjugate {\em transversal} lines in this cubic extension, denoted $g,g^q,g^{q^2}$, which meet each extended plane of $\S$, for more details on regular spreads and transversals, see \cite[Section 25.6]{HT}. 
 An $r$-space $\Pi_r$ of $\PG(6,q)$ lies in a unique $r$-space of $\PG(6,q^3)$, denoted $\Pi_r^\star$. 
A nrc $\N$ of   $\PG(6,q)$ lies in a unique nrc of $\PG(6,q^3)$, denoted $\N^\star$. Let $\var$ be a ruled quintic surface with conic directrix $\C$, twisted cubic directrix $\N_3$, and associated projectivity $\phi$. Then 
 we can extend $\var$ to a unique ruled quintic surface $\var^\star$ of $\PG(6,q^3)$
 with conic directrix $\C^\star$, twisted cubic directrix $\N_3^\star$, with the same associated projectivity, that is, extend $\phi$ from acting on $\PG(1,q)$ to acting on $\PG(1,q^3)$. We need the following characterisations. 
 
 \begin{result}\cite{FFA,iff} \Label{FFA-result}
 Let $\S$ be a regular 2-spread in a $5$-space $\si$ in $\PG(6,q)$ and consider the Bruck-Bose plane $\PG(2,q^3)$. 
\begin{enumerate}
\item An $\Fq$-subline  of $\PG(2,q^3)$  that meets $\li$ in a point corresponds in $\PG(6,q)$  to a line not in $\si$. 
\item An $\Fq$-subline   of $\PG(2,q^3)$ that is disjoint from $\li$  corresponds  in $\PG(6,q)$ to a  twisted cubic $\N_3$ lying in a 3-space about a plane of $\S$,  such that the  extension $\N_3^\star$ to $\PG(6,q^3)$ meets each transversal of $\S$  in a point.
\item An $\Fq$-subplane of $\PG(2,q^3)$  tangent to $\li$ at the point $T$ corresponds  in $\PG(6,q)$ to a ruled quintic surface $\var$ with conic directrix in the spread plane corresponding to $T$, such that in the cubic extension $\PG(6,q^3)$, the transversals $g,g^q,g^{q^2}$ of $\S$ are generators of ${\var}\star$.\end{enumerate}
Moreover, the converse of each is true.
\end{result}

We use this characterisation to show that $\var$ contains exactly $q^2$ twisted cubics.

\begin{theorem}\Label{tc-direct}
The ruled quintic surface  $\var$ contains exactly $q^2$ twisted cubics, each  is a directrix of $\var$. 
\end{theorem}

\begin{proof} 
By Theorem~\ref{the-nine-quad}, all ruled quintic surfaces are projectively equivalent. So without loss of generality, we can position a ruled quintic surface so that it corresponds to an $\Fq$-subplane of $\PG(2,q^3)$ which we denote by $\B$. That is, by Result~\ref{FFA-result}, $\S$ is a regular 2-spread in a hyperplane $\si$,  $\var\cap\si$ is  the conic directrix $\C$ of $\var$,   $\C$ lies in a plane of $\S$, and in the cubic extension $\PG(6,q^3)$, the transversals $g,g^q,g^{q^2}$ of $\S$ are generators of ${\var}^\star$.

  Let $\N_3$ be a twisted cubic contained in $\var$, and denote the 3-space containing $\N_3$ by $\Pi_3$. 
  As $\var\cap\si=\C$,  $\Pi_3$ meets $\si$ in a plane, we show this is a plane of $\S$. 
    In $\PG(6,q^3)$, $\var^\star$ is a ruled quintic surface that contains the twisted cubic $\N_3^\star$, moreover, the transversals $g,g^q,g^{q^2}$ of $\S$ are generators of $\var^\star$. So by Theorem~\ref{one-pt-gen}, $\N_3^\star$ contains one point on each of $g,g^q$ and $g^{q^2}$.  Hence the 3-space $\Pi_3^\star$ contains an extended plane of $\S$, and so $\Pi_3$ meets $\si$ in a plane of $\S$.
  Hence $\Pi_3\cap\alpha=\emptyset$, further, by Theorem~\ref{one-pt-gen}, $\N_3$ contains one point on each generator of $\var$, thus $\N_3$ is a directrix of $\var$. 
   
   By Result~\ref{FFA-result}, 
 $\N_3$ corresponds  in $\PG(2,q^3)$ to an $\Fq$-subline of $\B$ disjoint from $\li$. Conversely,  every $\Fq$-subline of $\B$ disjoint from $\li$ corresponds to a twisted cubic on $\var$. 
 Thus the twisted cubics in $\var$ are in 1-1 correspondence with the $\Fq$-sublines of  $\B$ that are disjoint from $\li$. 
 As there are $q^2$ such $\Fq$-sublines, there are $q^2$ twisted cubics on $\var$. 
 \end{proof}

%
%
%

Suppose we position $\var$ so that it corresponds via the Bruck-Bose representation  to a tangent $\Fq$-subplane $\B$ of $\PG(2,q^3)$. So we have a regular 2-spread $\S$ in a hyperplane $\si$, and the conic directrix of $\var$ lies in a plane $\alpha\in\S$. 
We define the {\em splash} 
of $\B$ to be the set of $q^2+1$ points on $\li$ that 
 lie on an extended line of $\B$. The {\em splash} 
 of $\var$ is defined to be the corresponding set of $q^2+1$ planes of $\S$. We denote the splash of $\var$ by $\mathbb S$. Note that $\alpha$ is a plane of $\mathbb S$. 
 We show that the remaining $q^2$ planes of $\mathbb S$ are related  to the $q^2$ twisted cubics of $\var$.

\begin{corollary}\Label{remark-BB} 
Let $\S$ be a regular $2$-spread in a hyperplane $\si$ of $\PG(6,q)$. 
Without loss of generality, we can position $\var$ so that it corresponds via the Bruck-Bose representation  to a tangent $\Fq$-subplane of $\PG(2,q^3)$. Then the conic directrix of $\var$  lies in a plane $\alpha\in\S$,   the $q^2$ 3-spaces containing a twisted cubic of $\var$ meet $\si$ in distinct planes of $\S$, and these planes together with $\alpha$ form the splash $\mathbb S$ of $\var$.
\end{corollary}

\begin{proof}
By Theorem~\ref{the-nine-quad}, all ruled quintic surfaces are projectively equivalent, so without loss of generality, let $\var$ be  positioned so that it corresponds to an $\Fq$-subplane $\B$ of $\PG(2,q^3)$ which is tangent to $\li$.  
 Let $b$ be an $\Fq$-subline of $\B$ disjoint from $\li$, so the extension of $b$ meets $\li$ in a point $R$ which lies in the splash of $\B$. By  Result~\ref{FFA-result}, $b$ corresponds in $\PG(6,q)$ to a twisted cubic of $\var$ which  lies in a 3-space that meets $\si$ in the plane of $\mathbb S$ corresponding to the point $R$. 
 \end{proof}
 
Using this Bruck-Bose setting, we describe the 3-spaces of $\PG(6,q)$ that contain a plane of the regular 2-spread $\S$.
 
 \begin{corollary}\Label{tc-splash-3}
 Position $\var$ as in Corollary~\ref{remark-BB}, so $\S$ is a regular 2-spread in the hyperplane $\si$, and the conic directrix of $\var$ lies in a plane $\alpha$ contained in the splash $\mathbb S\subset \S$ of $\var$. 
 \begin{enumerate}
 \item Let $\beta\in\mathbb S\setminus\alpha$, then there exists a unique 3-space containing $\beta$ that meets $\var$ in a twisted cubic. The remaining $3$-spaces containing $\beta$ (and not in $\si$) meet $\var$ in 0 or 1 point. 
 \item Let $\gamma\in\S\setminus\mathbb S$, then each 3-space containing $\gamma$ and not in $\si$ meets $\var$ in 0 or 1 point. 
 \end{enumerate}
 \end{corollary}
 
 \begin{proof} By Corollary~\ref{remark-BB}, we can position $\var$ so that it corresponds to an $\Fq$-subplane $\B$ of $\PG(2,q^3)$ which is tangent to $\li$. The 3-spaces that contain a plane of $\S$ (and do not lie in $\si$) correspond to lines of $\PG(2,q^3)$. Each point on $\li$  not in $\B$ but in the splash of $\B$  lies on a unique line that meets $\B$ in an $\Fq$-subline. By Result~\ref{FFA-result}, this corresponds to a twisted cubic in $\var$. The remaining lines meet $\B$ in 0 or 1 point, so the remaining 3-spaces meet $\var$ in 0 or 1 point.
\end{proof}

As $\var$ corresponds to an $\Fq$-subplane, we have the following result. 

\begin{theorem}\Label{2pts-tc}
Let $\var$  be a ruled quintic surface in $\PG(6,q)$. 
\begin{enumerate}
\item Two twisted cubics on $\var$ meet in a unique point.
\item Let  $P,Q$ be points lying on different generators of $\var$, and not in the conic directrix. Then $P,Q$ lie on a unique twisted cubic of $\var$. 
\end{enumerate}
\end{theorem}

\begin{proof} Without loss of generality let $\var$ be positioned as  described in Corollary~\ref{remark-BB}. So the conic directrix lies in a plane $\alpha$ contained in a regular 2-spread $\S$ in $\si$, and $\var$ corresponds to a  $\Fq$-subplane $\B$ of $\PG(2,q^3)$ tangent to $\li$. 
Let $\N_1,\N_2$ be two twisted cubics contained in $\var$. By Result~\ref{FFA-result}, they correspond in $\PG(2,q^3)$ to two $\Fq$-sublines of $\B$ not containing $\B\cap\li$, and so meet in a unique affine point  $P$. This corresponds to a unique point $P\in\var\setminus\alpha$ lying in both $\N_1$ and $\N_2$, proving part 1. 

For part 2, let $P,Q$ be points lying on distinct generators of $\var$, $P,Q\notin\C$. If the line $PQ$ met $\alpha$, then 
$\langle \alpha,P,Q\rangle$ is a 3-space that contains $\alpha$ and the generators of $\var$ containing $P$ and $Q$, contradicting Corollary~\ref{2-gen-alpha}. 
Hence the line  $PQ$ is skew to $\alpha$.  In $\PG(2,q^3)$, $P,Q$ correspond to two affine points in the tangent $\Fq$-subplane $\B$, so they lie on a unique $\Fq$-subline $b$ of $\B$. 
By Result~\ref{FFA-result}, the generators of $\var$ correspond to the $\Fq$-sublines of $\B$ through the point $\B\cap\li$. As $PQ$ is skew to $\alpha$, we have  $b\cap\li=\emptyset$. 
Hence by Result~\ref{FFA-result}, in $\PG(6,q)$, $P,Q$ lie on a unique twisted cubic of $\var$. 
 \end{proof}

\section{$5$-spaces meeting $\var$}\Label{sec-5}

In this section we determine how 5-spaces meet $\var$ and count the different intersection types. A series of lemmas is used to prove the main result which is stated in Theorem~\ref{count-nrc}.

\begin{lemma}\Label{5containsC}
 Let $\var$ be a ruled quintic surface of $\PG(6,q)$  with conic directrix $\C$. Of the 
$q^3+q^2+q+1$ $5$-spaces of $\PG(6,q)$ containing $\C$, $r_i$ meet $\var$ in precisely $\C$ and $i$ generators, where
$$r_3=\frac{q^3-q}6,\quad r_2=q^2+q,\quad r_1=\frac {q^3}2+\frac {q}2+1,\quad r_0=\frac{q^3-q}3.$$
%
\end{lemma}

\begin{proof}  Let $\var$ be a ruled quintic surface of $\PG(6,q)$  with conic directrix $\C$ lying in a plane $\alpha$.  By Lemma~\ref{5-space-quintic}, a 5-space containing $\C$ contains at most three generator lines of $\var$. 
By Theorem~\ref{x-gen}, three generators of $\var$ lie in a unique $5$-space. Hence there are 
$$r_3={q+1\choose3}$$
 5-spaces that contain three generators of $\var$. Such a 5-space contains three points of $\C$, and so contains $\C$ and $\alpha$.

Denote the generator lines of $\var$ by $\ell_0,\ldots,\ell_q$ and consider two generators,  $\ell_0,\ell_1$ say. By Corollary~\ref{2-gen-alpha}, $\Sigma_4=\langle \alpha,\ell_0,\ell_1\rangle$ is a 4-space. By  Theorem~\ref{x-gen}, $\langle\Sigma_4,\ell_i\rangle$ $i=2,\ldots, q$ are distinct 5-spaces. That is, $q-1$ of the 5-spaces about $\Sigma_4$ contain 3 generators, and hence the remaining two contain $\ell_0$, $\ell_1$ and no further generator of $\var$. Hence by Lemma~\ref{5-space-quintic}, $q-1$ of the 5-spaces about $\Sigma_4$ meet $\var$ in exactly $\C$ and  3 generators; and the remaining two 5-spaces about $\Sigma_4$ meet $\var$ in exactly  $\C$ and two generators. There are ${q+1\choose 2}$ choices for 
$\Sigma_4$, hence the number of 5-spaces that meet $\var$ in precisely  $\C$ and two generators is $$r_2=2\times {q+1\choose 2}=(q+1)q.$$ 
Next, let $r_1$ be the number of $5$-spaces that meet $\var$ in precisely $\C$ and one generator. We count in two ways ordered pairs $(\ell,\Pi_5)$ where $\ell$ is a generator of $\var$, and $\Pi_5$ is a 5-space that contains $\ell$ and $\alpha$, giving
$$(q+1)(q^2+q+1)=3r_3\ +\ 2r_2\ +\ r_1.$$
Hence $r_1=q^3/2+q/2+1$. 
Finally, the number of 5-spaces containing $\C$ and zero generators is $r_0=(q^3+q^2+q+1) - r_3-r_2-r_1=(q^3-q)/3$, as required.
\end{proof}

\begin{lemma}\Label{5containsN}
 Let $\var$ be a ruled quintic surface of $\PG(6,q)$ and let 
  $\N_3$ be a twisted cubic directrix of $\var$.
\begin{enumerate}
\item Of  the $q^2+q+1$ $5$-spaces of $\PG(6,q)$ containing $\N_3$, $s_i$ meet $\var$ in precisely $\N_3$ and $i$ generators, where
$$s_2=\frac{q^2+q}2,\quad s_1=q+1,\quad\quad s_0=\frac{q^2-q}2.$$
\item The total number of 5-spaces that meet $\var$ in a twisted cubic and $i$ generators is $q^2s_i$, $i=0,1,2$.
\end{enumerate}
\end{lemma}

\begin{proof}
Let $\var$ be a ruled quintic surface of $\PG(6,q)$
  with a twisted cubic directrix $\N_3$ lying in the 3-space $\Pi_3$.   
  By Lemma~\ref{5-space-quintic}, a 5-space containing $\N_3$ contains at most two generators  of $\var$, so the number of 5-spaces that contain $\Pi_3$ and exactly two generator lines is $s_2={q+1\choose 2}$. Let $\ell$ be a generator of $\var$ and consider the  4-space $\Pi_4=\langle \Pi_3,\ell\rangle$. 
 For each generator $m\neq\ell$, $\langle\Pi_4,m\rangle$ is a 5-space about $\Pi_4$ that meets $\var$ in $\N_3$, $\ell$ and $m$, and in no further point by Lemma~\ref{5-space-quintic}. 
 This accounts for $q$ of the 5-spaces containing $\Pi_4$. Hence the remaining 5-space containing $\Pi_4$  meets $\var$ in exactly $\N_3$ and $\ell$. 
 That is,  exactly  one of the 5-spaces about $\Pi_4=\langle \Pi_3,\ell\rangle$  meets $\var$ in precisely $\N_3$ and $\ell$. There are $q+1$ choices for the generator  $\ell$, hence $s_1=q+1$. Finally $s_0=(q^2+q+1)-s_2-s_1=(q^2-q)/2$, as required. For part 2, by Lemma~\ref{tc-direct}, $\var$ contains $q^2$ twisted cubics, so the total number of 5-spaces meeting $\var$ in a twisted cubic and $i$ generators is $q^2s_i$, $i=0,1,2$.
\end{proof}


%
%

The next result looks at properties of 4-dim nrcs contained in $\var$. In particular, we show that 
there are no 5-spaces that meet $\var$ in a 4-dim nrc and 0 generator lines.

\begin{lemma}\Label{5contains4}
 Let $\var$ be a ruled quintic surface of $\PG(6,q)$ with conic directrix $\C$ in the plane $\alpha$, and let 
  $\N_4$ be a $4$-dim nrc contained in $\var$. 
  \begin{enumerate}
    \item 
The $q+1$ $5$-spaces  containing $\N_4$ each contain a distinct generator line of $\var$.
\item The $4$-space containing $\N_4$ meets $\alpha$ in a point $P$, and either $P=\C\cap\N_4$ 
 or $q$ is even and $P$ is the nucleus of $\C$. 
\end{enumerate}
\end{lemma}

\begin{proof}
Let $\var$ be a ruled quintic surface in $\PG(6,q)$ with conic directrix $\C$ lying in a plane $\alpha$. 
 Let $\N_4$ be a 4-dim nrc contained in $\var$, so $\N_4$ lies in a 4-space which we denote $\Pi_4$.
 By Corollary~\ref{3-4-no-gen}, $\Pi_4$ does not contain a generator of $\var$.
 By Lemma~\ref{5-space-quintic}, a 5-space containing $\N_4$ can contain at most one generator of $\var$. Hence each of the $q+1$ 5-spaces containing $\N_4$ contains a distinct generator.  In particular, if we label the points of $\C$ by $Q_0,\ldots,Q_{q}$, and the generator through $Q_i$ by $\ell_{Q_i}$, then the $q+1$ 5-spaces containing $\N_4$ are $\Sigma_i=\langle\Pi_4,\ell_{Q_i}\rangle$, $i=0,\ldots,q$. 
 
If $\Pi_4$ met the plane $\alpha$  in a line, then $\langle\Pi_4,\alpha\rangle$ is a 5-space whose intersection with  $\var$ contains $\N_4$ and $\C$, contradicting Lemma~\ref{5-space-quintic}. Hence $\Pi_4$ meets $\alpha$ in a point $P$.  
 There are three possibilities for the point $P=\Pi_4\cap\alpha$, namely $P\in\C$; $q$ even and $P$ the nucleus of $\C$; or $q$ even, $P\notin\C$, and $P$ not the nucleus of $\C$.

Case 1, suppose  $P\in\C$. 
For $i=0,\ldots,q$,  the 5-space $\Sigma_i=\langle\Pi_4,\ell_{Q_i}\rangle$ meets $\alpha$ in a line $m_i$. Label $\C$ so that $P=Q_0$,  so the line $m_0$ is the tangent to $\C$ at $P$, and  $m_i$, $i=1,\ldots,q$, is the secant line $ PQ_i$. We now show that $P=Q_0$ is a point of $\N_4$. 
Let $i\in\{1,\ldots,q\}$, then by Lemma~\ref{5-space-quintic}, $\Sigma_i$ meets $\var$ in precisely $\N_4\cup \ell_{Q_i}$, and $\Sigma_i\cap\var\cap\alpha$ is the two points $P,Q_i$. As $P\notin\ell_{Q_i}$ we have $P\in\N_4$. That is, $P=\C\cap\N_4$.

Case 2, suppose $q$ is even and $P=\Pi_4\cap\alpha$ is the nucleus of $\C$.
For $i=0,\ldots,q$, the 5-space $\Sigma_i=\langle\Pi_4,\ell_{Q_i}\rangle$ meets $\alpha$ in the tangent to $\C$ through $Q_i$. In this case, $\C\cap\N_4=\emptyset$.

Case 3, suppose $P=\Pi_4\cap\alpha$ is not in $\C$, and $P$ is not the nucleus of $\C$. Now $P$ lies on some secant $m=QR$ of $\C$, for some points $Q,R\in\C$.
The intersection of the  5-space $\langle \Pi_4,m\rangle$ with $\var$ contains $\N_4$ and two points $R,Q$ of $\C$. 
As $R,Q$ lie on distinct generators and are not in $\N_4$, this contradicts Lemma~\ref{5-space-quintic}.  
Hence this case cannot occur. 
\end{proof}

 We can now  describe how a nrc of $\var$ meets the conic directrix, and note that Theorem~\ref{count-plane-line} shows that each possibility in part 3 can occur. 

\begin{corollary}\Label{N-r-on-V} 
Let $\var$ be a ruled quintic surface of $\PG(6,q)$ with conic directrix $\C$.
\begin{enumerate}
\item  A twisted cubic $\N_3\subseteq \var$ contains 0 points of $\C$.
\item  A $4$-dim nrc
$\N_4\subseteq\var$ either (i) contains 1 point of $\C$; or (ii) contains 0 points of $\C$, in which case  $q$ is even and the 4-space containing $\N_4$ contains  the nucleus of $\C$.
\item  A $5$-dim nrc
$\N_5\subseteq\var$  contains 0 or 1 or 2 points of $\C$.
\end{enumerate}
\end{corollary}

\begin{proof} Let $\var$ be a ruled quintic surface of $\PG(6,q)$ with conic directrix $\C$ in a plane $\alpha$. Let $\N_3$ be a twisted cubic of $\var$, so by 
Lemma~\ref{tc-direct},  $\N_3$ is a directrix of $\var$, and so is disjoint from $\alpha$, proving part 1. Next let $\N_4$ be a 4-dim nrc on $\var$, and let $\Pi_4$ be the 4-space containing $\N_4$. 
 By Lemma~\ref{5contains4}, $\Pi_4\cap\alpha$ is a point $P$, and either $P=\C\cap\N_4$; or $q$ is even and $P$ is the nucleus of $\C$, and so $P\notin\var$, hence $P\notin\N_4$, proving part 2.
  Let $\Pi_5$ be a 5-space containing a 5-dim nrc of $\var$. By Lemma~\ref{5-space-quintic}, $\Pi_5$ cannot contain $\alpha$. Hence $\Pi_5$ meets $\alpha$ in a line, and so contains at most two points of $\C$, proving part 3.
\end{proof}

We now use the Bruck-Bose setting  to count the 4-dim nrcs contained in $\var$.

%
%
%

\begin{lemma} \Label{part-of-next-lemma}
Let $\S$ be a regular 2-spread in a 5-space $\si$ in $\PG(6,q)$. 
Position $\var$ as in Corollary~\ref{remark-BB}, so $\var$ has splash $\mathbb S\subset\S$.  
Then a 5-space/4-space about a plane $\beta\in\mathbb S$ cannot contain a 4-dim nrc of $\var$.
\end{lemma}

\begin{proof} 
Position $\var$ as described in Corollary~\ref{remark-BB}, so 
 $\S$ is a regular 2-spread in a 5-space $\si$,   the conic directrix of $\var$ lies in a plane $\alpha\in\S$, and $\mathbb S\subset\S$ denotes the  splash of $\var$. By Lemma~\ref{5-space-quintic}, a 4-space containing $\alpha$ cannot contain a 4-dim nrc of $\var$. 
Let  $\beta\in\mathbb S\setminus\alpha$, then by Corollary~\ref{tc-splash-3}, $\beta$ lies in exactly one 3-space  that contains a twisted cubic of $\var$, denote these by $\Pi_3$ and $\N_3$ respectively.
By Theorem~\ref{tc-direct}, $\N_3$ is a directrix of $\var$, and so  $\Pi_3$ is disjoint from $\alpha$. 
So if  $\ell_P$ is a generator  of $\var$, then $\Pi_4=\langle\Pi_3,\ell_P\rangle$ is a 4-space and   $\Pi_4\cap\alpha$ is the point $P=\ell_P\cap\C$. Let $\ell$ be a line of $\alpha$ through $P$ and let $\Pi_5=\langle\Pi_3,\ell\rangle$. If $\ell$ is tangent to $\C$, then $\Pi_5\cap\var$ is exactly $\N_3\cup\ell_P$. If $\ell$ is a secant of $\C$, so $\ell\cap\C=\{P,Q\}$, then $\Pi_5\cap\var$ consists of $\N_3,\ell_P$ and the generator  $\ell_Q$ through $Q$. 
Varying $\ell_P$ and $\ell$, we get all the 5-spaces that contain $\beta$ and contain 1 or 2 generators of $\var$. That is, each 5-space containing $\beta$ and 1 or 2 generators of $\var$ also contains $\N_3$. The remaining 5-spaces about $\beta$ hence contain 0 generators of $\var$ and  meet $\alpha$ in an exterior line of $\C$. 
Hence by  Lemma~\ref{5contains4}, none of the 5-spaces about $\beta$  contain a 4-dim nrc of $\var$. 
\end{proof}

\begin{lemma}\Label{count-4-nrc}
\begin{enumerate}
\item The number of 4-dim nrcs contained in $\var$ is $q^4-q^2$. 
\item 
The number of 5-spaces that meet $\var$ in a 4-dim nrc and one generator is 
$q^5+q^4-q^3-q^2$.
\end{enumerate}
\end{lemma}

\begin{proof}
Without loss of generality  position $\var$ as described in Corollary~\ref{remark-BB}. 
That is, let $\S$ be a regular 2-spread in a 5-space $\si$,  let the conic directrix of $\var$ lie in a plane $\alpha\in\S$, and let $\mathbb S\subset\S$ be the  splash of $\var$.
Straightforward counting shows that  a 5-space distinct from $\si$ contains a unique spread plane. 
If this plane is in the splash $\mathbb S$, then by Lemma~\ref{part-of-next-lemma}, the 5-space does not  contain a 4-dim nrc of $\var$. So a 5-space containing a 4-dim nrc of $\var$ contains a unique plane of $\S\setminus\mathbb S$. 
 Consider a plane $\gamma\in\S\setminus\mathbb S$. Let $P\in\C$, let $\ell_P$ be the  generator  of $\var$ through $P$, and consider the 4-space $\Pi_4=\langle\gamma,\ell_P\rangle$. 
Suppose first that  $\Pi_4$ contains two generators of $\var$, then there is a 5-space $\Pi_5$ containing $\gamma$ and two generators. By Lemma~\ref{5-space-quintic}, $\Pi_5$ contains either $\C$ or a twisted cubic of $\var$. A 5-space distinct from $\si$ cannot contain two planes of $\S$, so $\Pi_5$ does not contain $\C$. Moreover, by Corollary~\ref{remark-BB}, $\Pi_5$ does not contain a twisted cubic of $\var$. Hence $\Pi_4$ contains exactly one generator of $\var$. 
If every generator of $\var$ contained at least one point of $\Pi_4$, then the intersection of $\Pi_4$ with $\var$ contains at least $\ell_P$ and $q$ further points, one on each generator. By Lemma~\ref{5-space-quintic} and Corollary~\ref{3-4-no-gen}, the only possibility is that $\Pi_4\cap\var$ contains a twisted cubic, which is not possible by Corollary~\ref{remark-BB}.
Hence there is at least one generator which is disjoint from $\Pi_4$, denote this $\ell_Q$. Label the points of $\ell_Q$ by $X_0,\ldots,X_q$, then the $q+1$ 5-spaces containing $\Pi_4$ are $\Sigma_i=\langle\gamma,\ell_P,X_i\rangle$. For each $i=0,\ldots,q$, the intersection of $\Sigma_i$ with  $\var$ contains the generator $\ell_P$ and the point $X_i$. By 
 Corollary~\ref{remark-BB}, $\Sigma_i$ 
 does not contain a twisted cubic of $\var$. Hence  by 
  Lemma~\ref{5-space-quintic}, $\Sigma_i\cap\var$ is $\ell_P$ and a 4-dim nrc.
  
That is, there are $(q+1)^2$ 5-spaces containing $\gamma$ and one generator of $\var$, each contains a 4-dim nrc of $\var$. 
Further, if $\Pi_5$ is a 5-space containing $\gamma$ and zero generators of $\var$, then by Lemma~\ref{5contains4}, $\Pi_5$ does not contain a 4-dim nrc of $\var$. Hence as there are $q^3-q^2$ choices for $\gamma$, there are  $(q+1)^2\times(q^3-q^2)=q^5+q^4-q^3-q^2 $ 5-spaces that meet $\var$ in one generator and a 4-dim nrc. By Lemma~\ref{5contains4}, every 4-dim nrc in $\var$ lies in $q+1$ such  5-spaces. Hence the number of 4-dim nrcs contained in $\var$ is $(q^5+q^4-q^3-q^2)/(q+1)$ as required. 
\end{proof}

%
%
%
%
%

We now count the number of 5-dim nrcs contained in $\var$. 

\begin{lemma}\Label{no-5-5nrc}
 The number of 5-spaces meeting $\var$ in a 5-dim nrc is $q^6-q^4$.
 \end{lemma}
 
 \begin{proof}
We show that the number of 5-spaces meeting $\var$ in a 5-dim nrc is $q^6-q^4$ by counting in two ways the number $x$ of  incident pairs $(A,\Pi_5)$ where $A$ is a point of $\var$, and $\Pi_5$ is a 5-space containing $A$. 
The number of ways to choose a point $A$  of $\var$ is 
$(q+1)^2.$
The point $A$   lies in 
$q^5+q^4+q^3+q^2+q+1$ 5-spaces. So 
$$x= (q+1)^2\times (q^5+q^4+q^3+q^2+q+1)=q^7+3q^6+4q^5+4q^4+4q^3+4q^2+3q+1.$$
Alternatively, we count the 5-spaces first; there are several possibilities for $\Pi_5$. By Lemmas~\ref{5-space-quintic}, $\Pi_5\cap\var$ is either empty, or contains an $r$-dim nrc, for some $r\in\{2,\ldots,5\}$. Let $n_r$ be the number of pairs $(A,\Pi_5)$ with $A\in\var\cap\Pi_5$ and $\Pi_5$ containing an $r$-dim nrc of $\var$. Note that 
\begin{eqnarray}\label{x-eqn}
x=n_2+n_3+n_4+n_5.
\end{eqnarray}
We now calculate  $n_2$, $n_3$ and $n_4$,  and then use (\ref{x-eqn}) to determine the number of 5-spaces meeting $\var$ in a 5-dim nrc.  
\begin{enumerate}

\item Consider a 5-space $\Pi_5$ that contains the conic directrix $\C$, so by Lemma~\ref{5containsC}, $\Pi_5$ contains 0, 1, 2 or 3 generators of $\var$, and the number of 5-spaces  meeting $\var$ in  exactly the conic directrix and $i$ generators  is $r_i$. In this  case  the number of ways to pick a point of $\Pi_5\cap\var$ is $iq+q+1$.
 Hence the total number of pairs  $(A,\Pi_5)$ with $\Pi_5$ containing the conic directrix is 
 $$n_2=\sum_{i=0}^3 r_i(iq+q+1)= 2q^4+4q^3+4q^2+3q+1.$$

\item Consider a 5-space $\Pi_5$ that contains a twisted cubic, then by Lemma~\ref{5containsN}, $\Pi_5$ contains 0, 1 or 2  generators of $\var$, and the number of 5-spaces    meeting $\var$ in   a given twisted cubic and $i$ generators  is $s_i$. In this  case the number of ways to pick $A$ in $\var\cap\Pi_5$  is $iq+q+1$.
 Hence the number of   pairs $(A,\Pi_5)$ with $\Pi_5$  containing a twisted cubic of $\var$ is $$n_3=q^2\sum_{i=0}^2 s_i(iq+q+1)= 2q^5+4q^4+3q^3+q^2.$$

\item Consider a 5-space $\Pi_5$ that contains a 4-dim nrc of $\var$. By Lemma~\ref{5contains4}, $\Pi_5$  contains 1  generator of $\var$.  By Lemma~\ref{count-4-nrc}, the number of 5-spaces  meeting $\var$ in  exactly a 4-dim nrc and one generator is  $q^5+q^4-q^3-q^2$.
 The number of ways to pick $A$ in $\var\cap\Pi_5$ is $2q+1$.
So $$n_4= (q^5+q^4-q^3-q^2) \times (2q+1)=2q^6+3q^5-q^4-3q^3-q^2.$$

\item Denote the number of $5$-spaces containing a 5-dim nrc  of $\var$ by $y$. Then the number of   pairs $(A,\Pi_5)$ with $\Pi_5$  containing a 5-dim nrc of $\var$  is 
$$n_5=y\times (q+1).$$

\end{enumerate}

Substituting the calculated values for $x,n_2,n_3,n_4,n_5$ into (\ref{x-eqn}) and rearranging gives $y=q^6-q^4$ as required. 
\end{proof}

%
%
%

Summarising the preceding  lemmas gives the 
 following  theorem  describing  $\var$. 

\begin{theorem}\Label{count-nrc}
 Let $\var$ be the ruled quintic surface in $\PG(6,q)$, $q\geq 6$.
 \begin{enumerate}
\item $\var$ contains exactly
 \begin{center}
\begin{tabular}{rl}
$q+1$ & lines,
\\ 1 &non-degenerate conic,
\\ $q^2$& twisted cubics,
\\ $q^4-q^2$& 4-dim nrcs,
\\ $q^6-q^4$ &5-dim nrcs.
\end{tabular}\end{center}
\item  A 5-space meets $\var$ in one of the following configurations
\begin{center}
\begin{tabular}{r|l}
Number of $5$-spaces&meeting $\var$ in the configuration\\
\hline
$q^6-q^4$      & 5-dim nrc\\ 
$q^5+q^4-q^3-q^2$       & 4-dim nrc and 1 generator \\
$(q^4-q^3)/2$ & twisted cubic \\
 $ q^3+q^2$     &  twisted cubic  and 1 generator \\
$(q^4+q^3)/2$      &  twisted cubic  and 2 generators\\
 $(q^3-q)/3$ & conic \\
  $q^3/2+q/2+1$    & conic  and 1 generator \\
 $ q^2+q$     & conic  and 2 generators\\
  $(q^3-q)/6$    & conic  and 3 generators.\\
%
 \end{tabular}
 \end{center}
 \end{enumerate}
\end{theorem}

\section{$5$-spaces and the Bruck-Bose spread}\Label{sec-BB-5}

Let $\S$ be a regular 2-spread in a 5-space $\si$ in $\PG(6,q)$, and position $\var$ so that it corresponds to a tangent  $\Fq$-subplane of $\PG(2,q^3)$. So  $\var$ has splash $\mathbb S\subset\S$,  the conic directrix $\C$ lies in a plane $\alpha\in\mathbb S$, and each of the $q^2$ 3-spaces containing a twisted cubic directrix of $\var$ meets $\si$ in a distinct plane of $\mathbb S\setminus\alpha$.  
In Corollary~\ref{tc-splash-3}, we looked at how 3-spaces containing a plane of $\S$ meet $\var$. In Lemma~\ref{part-of-next-lemma}, we looked at how  4-spaces containing a plane of $\S$ meet $\var$. Next we look at how 
5-spaces containing a plane of $\S$ meet $\var$. Note that
straightforward counting shows that  a 5-space distinct from $\si$ contains a unique plane $\pi$ of $\S$, and meets every other plane of $\S$ in a line. If  $\pi=\alpha$, then Lemma~\ref{5containsC} describes the possible intersections with $\var$. The next theorem describes the possible intersections with $\var$ for the remaining  cases $\pi\in\mathbb S\setminus\alpha$ and $\pi\in\S\setminus\mathbb S$.

\begin{theorem}\Label{count-plane-line}
Position $\var$ as in Corollary~\ref{remark-BB}, so $\S$ is a regular 2-spread in a hyperplane $\si$, the conic directrix $\C$ lies in a plane $\alpha\in\S$ and $\var$ has splash $\mathbb S\subset\S$. Let $\ell$ be a line of $\alpha$ with $|\ell\cap\C|=i$ and let $\pi\in\S$, $\pi\neq\alpha$. Then the $q$ 5-spaces containing $\pi,\ell$ and distinct from $\si$ meet $\var$ as follows.
\begin{enumerate}
\item If $\pi\in\mathbb S\setminus\alpha$, then $q-1$ meet $\var$ in a 5-dim nrc, and 1 meets $\var$ in a twisted cubic and $i$ generators.
\item If $\pi\in\S\setminus\mathbb S$, then $q-i$ meet $\var$ in a 5-dim nrc, and $i$ meet $\var$ in a 4-dim nrc and $1$ generator.
\end{enumerate}
\end{theorem}

\begin{proof}
By \cite{FFA}, the group  of collineations of $\PG(6,q)$ fixing $\S$ and $\var$ is transitive on the planes of $\mathbb S\setminus\alpha$ and on the planes of $\S\setminus\mathbb S$. As this group fixes the conic directrix $\C$, it is transitive on the lines of $\alpha$ tangent to $\C$,  
 the lines of $\alpha$ secant to $\C$ and 
 the lines of $\alpha$ exterior to $\C$.  
So without loss of generality let $\ell_0$  be a line of $\alpha$ exterior to  $\C$,  let $\ell_1$  be a line of $\alpha$ tangent to  $\C$, let $\ell_2$  be a line of $\alpha$ secant to  $\C$,  let $\beta$ be a plane  in $\mathbb S\setminus\alpha$, and  let $\gamma$ be a plane of $\S\setminus\mathbb S$. For $i=0,1,2$, label  the 4-spaces 
 $\Sigma_{4,i}=\langle\beta,\ell_i\rangle$ and $\Pi_{4,i}=\langle\gamma,\ell_i\rangle$. 
 By Corollary~\ref{tc-splash-3}, as $\beta\in\mathbb S\setminus\alpha$, there is a unique twisted cubic of $\var$ that lies in a 3-space about $\beta$, denote this 3-space by $\Pi_3$. Hence for $i=0,1,2$, there is a unique 5-space containing $\Sigma_{4,i}$ whose intersection with $\var$ contains a twisted cubic, namely the 5-space $\langle\Pi_3,\ell_i\rangle$.

First consider the line $\ell_0$ which is  exterior to $\C$. 
A 5-space meeting $\alpha$ in $\ell_0$ contains 0 points of $\C$, and so   contains 0 generators of $\var$.  The 4-space
  $\Sigma_{4,0}=\langle \beta,\ell_0\rangle$ lies in $q$ 5-spaces distinct from  $\si$,  each containing 0 generators of $\var$. 
 Exactly one of these 5-spaces, namely $\langle \Pi_3,\ell_0\rangle$,  contains a   
 twisted cubic of $\var$.
 The remaining $q-1$ 5-spaces about $\Sigma_{4,0}$ contain 0 generators, and do not contain a conic or twisted cubic of $\var$, so by Theorem~\ref{count-nrc}, they  meet $\var$ in a 5-dim nrc, proving part 1 for $i=0$.  For part 2, let $\Pi_5\neq\si$ be any 5-space containing  $\Pi_{4,0}=\langle \gamma,\ell_0\rangle$. 
As $\gamma\notin\mathbb S$, by Corollary~\ref{remark-BB}, $\Pi_5$ cannot contain a twisted cubic of $\var$. 
As $\Pi_5$ contains 0 generator lines of $\var$ and does not contain a conic or twisted cubic of $\var$, by Theorem~\ref{count-nrc}, $\Pi_5$ meets $\var$ in a 5-dim nrc. That is, the $q$ 5-spaces (distinct from $\si$) containing $\Pi_{4,0}$ meet $\var$ in a 5-dim nrc, proving part 2 for $i=0$.

Next consider the line $\ell_1$ which is  tangent to $\C$.  Let $P=\ell_1\cap\C$ and denote  the generator of $\var$ through $P$ by $\ell_P$. 
A 5-space meeting $\alpha$ in a tangent  line contains 1 point of $\C$, and so   contains at most one generator of $\var$.  So exactly one 5-space contains $\Sigma_{4,1}$ and a generator, namely the 5-space $\langle\Sigma_{4,1},\ell_P\rangle$. 
Consider the 5-space $\langle \Pi_3,\ell_1\rangle$, it  contains $P$ and a twisted cubic of $\var$ which by  Corollary~\ref{N-r-on-V} is disjoint from $\alpha$, hence $\langle \Pi_3,\ell_1\rangle$ contains the generator $\ell_P$. That is, 
$\langle \Pi_3,\ell_1\rangle$ contains $\beta$, $\ell_1$, $\ell_P$ and so 
  $\langle \Pi_3,\ell_1\rangle=\langle \Sigma_{4,1},\ell_P\rangle$. That is,   the intersection of $\langle \Sigma_{4,1},\ell_P\rangle$ with $\var$ is  a twisted cubic and one generator. 
Let $\Pi_5\neq \si$ be one of the remaining $q-1$ 5-spaces  (distinct from $\si$) that contains $\Sigma_{4,1}$, so $\Pi_5$ contains 0 generators of $\var$ and does not contain a conic or twisted cubic of $\var$.  So by Theorem~\ref{count-nrc}, $\Pi_5$ meets $\var$ in a 5-dim nrc,  
proving part 1 for $i=1$. For part 2, we  consider $\Pi_{4,1}=\langle \gamma,\ell_1\rangle$. By Corollary~\ref{remark-BB},  as $\gamma\notin\mathbb S$, no 5-space containing $\Pi_{4,1}$  contains a twisted cubic of $\var$. The 5-space $\langle\Pi_{4,1},\ell_P\rangle$ contains one generator of $\var$,  so by Theorem~\ref{count-nrc}, it meets $\var$ in exactly a 4-dim nrc and the generator $\ell_P$. Let $\Pi_5\neq\si$ be one of the remaining $q-1$ 5-spaces containing $\Pi_{4,1}$, then $\Pi_5$ contains 0 generators of $\var$.   So by Theorem~\ref{count-nrc}, $\Pi_5$ meets $\var$ in a 5-dim nrc,  
proving part 2 for $i=1$.

Finally, consider the line $\ell_2$ which is  secant to $\C$.    Let $\C\cap\ell_2=\{P,Q\}$ and let $\ell_P,\ell_Q$ be the generators of $\var$ through $P,Q$ respectively.  The intersection of the  5-space $\langle \Pi_3,\ell_2\rangle$ and $\var$ contains  a twisted cubic, and $P$ and $Q$. By Corollary~\ref{N-r-on-V}, this twisted cubic is disjoint from $\alpha$, so $\langle\Pi_3,\ell_2\rangle$  contains the two generators $\ell_P,\ell_Q$. Thus
$\langle \Pi_3,\ell_2\rangle=\langle\Sigma_{4,2},\ell_P\rangle=\langle\Sigma_{4,2},\ell_Q\rangle=\langle\Sigma_{4,2},\ell_P,\ell_Q\rangle$. 
The remaining $q-1$ 5-spaces  (distinct from $\si$) about $\Sigma_{4,2}$ contain 0 generators and two points of $\C$. By Lemma~\ref{5contains4} they cannot contain a 4-dim nrc of $\var$. So by Theorem~\ref{count-nrc}, they meet $\var$ in a 5-dim nrc, proving part 1 for $i=2$.
For part 2, let $\Pi_5\neq\si$ be a 5-space containing $\Pi_{4,2}=\langle\gamma,\ell_2\rangle$.
By Corollary~\ref{remark-BB}, $\Pi_5$ does not contain a twisted cubic of $\var$, as $\gamma\notin\mathbb S$. 
So by Theorem~\ref{count-nrc}, $\Pi_5$ contains at most one generator of $\var$. Hence $\langle \Pi_{4,2},\ell_P\rangle$, $\langle \Pi_{4,2},\ell_Q\rangle$ are distinct 5-spaces about $\Pi_{4,2}$, and by Theorem~\ref{count-nrc}, they each meet $\var$ in a 4-dim nrc and one generator. 
Let $\Sigma_5\neq\si$ be one of the 
 remaining $q-2$ 5-spaces about $\Pi_{4,2}$. Then $\Sigma_5$ contains 0 generators of $\var$, and so  by Theorem~\ref{count-nrc}, meets $\var$ in a 5-dim nrc, proving part 2 for $i=2$.
\end{proof}

\end{document}